\newif\ifsocg
\socgfalse

\ifsocg
\documentclass[a4paper,UKenglish,cleveref, autoref, thm-restate]{socg-lipics-v2021}
\usepackage{hyperref}

\title{Pach's animal problem within the bounding box}

\author{Martin Tancer}{Department of Applied Mathematics, Charles University, 
Malostransk\'{e} n\'{a}m. 25, 118~00~~Praha~1, 
Czech Republic}{}{}{Supported by the GA\v{C}R grant no. 22-19073S.}

\authorrunning{M. Tancer}

\Copyright{Martin Tancer}



\else
\documentclass[10pt,a4paper]{article}
\usepackage{a4wide}
\usepackage{fullpage}

\usepackage[utf8]{inputenc}

\usepackage{hyperref}
\usepackage{latexsym}
\usepackage{amssymb,amsmath,amsthm}
\usepackage{graphicx}
\usepackage{url}
\usepackage{color}
\usepackage{enumerate}
\usepackage[capitalize]{cleveref}
\usepackage{authblk}

\newtheorem{theorem}{Theorem}
\newtheorem*{theorem*}{Theorem}

\newtheorem{lemma}[theorem]{Lemma}

\theoremstyle{definition}

\theoremstyle{remark}

\makeatletter
\newcommand{\ProofEndBox}{{\ifhmode\unskip\nobreak\hfil\penalty50 \else
          \leavevmode\fi\quad\vadjust{}\nobreak\hfill$\Box$
            \finalhyphendemerits=0 \par}}
\makeatother

\fi

\newcommand{\R}{{\mathbb{R}}}

\newcommand{\Z}{{\mathbb{Z}}}

\newcommand{\bequal}{$\raisebox{1pt}{$\kern1.5mm
  {\scriptscriptstyle\circ}\kern-2.1mm\rule[-.4mm]{3mm}{.5pt}\kern-3mm\rule[1.2mm]{3mm}{.5pt}\kern1mm$}$}

\newcounter{sideremark}

\ifsocg\else
\title{Pach's animal problem within the bounding box\thanks{Supported by the GA\v{C}R
grant no. 22-19073S.}}
\author[1]{Martin Tancer}

\affil[1]{\small Department of Applied Mathematics, Charles University, Malostransk\'{e} n\'{a}m.
25, 118~00~~Praha~1, Czech Republic}

\date{}
\fi

\begin{document}
\maketitle
\begin{abstract}
  A collection of unit cubes with integer coordinates in $\R^3$ is an
  \emph{animal} if its union is homeomorphic to the 3-ball. Pach's animal
  problem asks whether any animal can be transformed to a single cube by adding
  or removing cubes one by one in such a way that any intermediate step is an 
  animal as well. Here we provide an example of an animal that cannot be
  transformed to a single cube this way within its bounding box.
\end{abstract}

\section{Introduction}
\subparagraph{Pach's animal problem.}
A \emph{grid cube} is a subset of $\R^3$ that can be written as $[a,a+1] \times
[b,b+1] \times [c,c+1]$ where $a$, $b$ and $c$ are integers. A \emph{grid
complex} is a 3-dimensional polytopal complex\footnote{By a \emph{polytopal
complex} we mean a collection of polytopes in $\R^d$ for some $d$ (in our case
$d=3$) such that (i) every face of any polytope in the collection belongs to the
collection and (ii) an intersection of two polytopes in the collection is a
face of both.} formed by a finite collection of
grid cubes and the faces of the cubes in this collection. A grid complex is an
animal if the union of cubes in the complex (i. e., the \emph{polyhedron} of
the complex) is homeomorphic to a 3-ball. In 1988 Pach asked whether
any animal can be transformed to a single cube by adding or removing cubes one
by one in such a way that any intermediate step is an animal as
well~\cite{orourke88}. This question is known as \emph{Pach's animal problem}
and has been reproduced in several other venues (including SoCG); see,
e.g., notes in Chapter 8 of~\cite{ziegler95} or~\cite{dumitrescu-pach04, orourke11}.  In the following text,
when we consider cube removals or additions, we always mean that each
intermediate step is an animal.

Surprisingly, this innocent-looking question is actually very complex and
resistant. On the
one hand, there are examples of animals that cannot be transformed to a single
cube by removals only: The first one (the author is aware of) is Furch's
``knotted hole ball'' from~1924~\cite{furch24} (see also~\cite{ziegler98}).
Another one from 1964 is a 3-dimensional variant of famous Bing's house with two
rooms~\cite{bing64}.\footnote{The aim of the constructions of Furch and Bing is
to obtain so called non-shellable balls. But for an animal non-shellable
exactly means that the animal cannot be transformed to a single cube by
removals only.} After Pach asked about the animal problem, Shermer
obtained particularly small such animals independently of the earlier
results~\cite{shermer88}. (See also~\cite{orourke11}.) On the other hand,
allowing also cube additions adds much more flexibility how to transform
animals. If we replace ``cube removals'' and ``cube additions'' with closely
related ``collapses'' and ``anticollapses'', it follows from a classical result of
Whitehead~\cite{whitehead50} that any animal can be reduced to a point (or a
cube) by collapses and anticollapses. But the integer grid does not seem to be
flexible enough to emulate all possible anticollapses. Altogether, adding
geometric restrictions coming from the integer grid to classical setting in
topology makes the question interesting.

In 2010, Nakamura wrote a technical report claiming a solution of Pach's animal
problem~\cite{nakamura10}. This technical report neither appeared in a print
nor it has been verified by the community. The author of this work believes
that Nakamura's proof contains significant gaps. Thus Pach's animal
problem should be still regarded as open. We will comment on this in more
detail in the appendix.

\subparagraph{Pach's animal problem within the bounding box.}
For all the aforementioned examples, even if they cannot be transformed to a single cube
by removals, it is extremely easy to transform them to a single cube if we also
allow additions of cubes. 
All the aforementioned examples can be built by gradual removals of cubes from
the \emph{bounding box} (i. e., the smallest grid-aligned box containing the
animal) while each intermediate step is an animal. If we revert this process, each of the aforementioned examples can be
transformed to the bounding box (by cube additions) and then to a single cube
(by cube removals). 

Dumitrescu and Hilscher~\cite{dumitrescu-hilscher11} provided an example of an
animal which cannot be transformed to the bounding box by cube additions but
this example is essentially the complement of Shermer's construction. Thus the
cost is that this animal can be easily transformed to a single cube by
removals. 

In principle it should be possible to combine (and possibly iterate) two types
of the aforementioned constructions which would require alternating cube
additions and cube removals in order to transform the animal into a single cube
within the bounding box. But this would still leave the hope that there is an
algorithm for Pach's animal problem which gradually simplifies the
``innermost'' part of the animal (or its complement) eventually reaching a
single cube.

We provide a new significantly stronger construction showing that this hope is vain.

\begin{theorem}
  \label{t:main}
  There is an animal $A$ such that it cannot be transformed to a single cube by
  additions or removals of cubes which are inside the bounding box of $A$. In
  fact, if we remove a cube from $A$ or add a cube to $A$ 
  contained inside the bounding box, we never obtain an animal.
\end{theorem}

Part of our motivation for proving Theorem~\ref{t:main} is also that we find it
realistic that this construction would be a part of a construction of a
counterexample to original Pach's animal problem (without any restriction coming from
the bounding box), of course, only if such a counterexample
exists.\footnote{The construction of $A$ shows a way to block the
``interior'' of an animal. (In this context the interior should be understood
loosely as a collection of cubes that are in the interior of the convex hull of
the animal.) The ``boundary'' cubes cannot be fully blocked. For example, a
cube can be always added to a top of a topmost cube. However it could be
realistic to control the ``expansion'' of the boundary using some ``undilatable
patterns'' (discussed in~\cite{nakamura10}) so that the cubes in the
``interior'' stay blocked no matter what is the expansion of the boundary.}

Another (independent) part of the motivation is that staying within the
bounding box is a natural restriction from point of view of digital geometry.
Here Pach's animal problem fits into a framework of deformations studied in 
digital geometry; see~\cite[Chapter~16]{klette-rosenfeld04}. (This is also a language used in Nakamura's technical
report~\cite{nakamura10}.) If cubes correspond to ``voxels'' stored in a computer,
then providing a bounding box is natural. We may also be interested in an
algorithm (possibly a brute force algorithm) transforming one animal into another (if this is possible). 
Our example shows that we have to leave the bounding box and it would be
interesting to know whether there is any (computable) bound on the space
required for such a transformation.\footnote{The author is strongly persuaded
that the animal constructed in this paper can be transformed into a single cube
if we allow leaving the bounding box. Thus our $A$ should not be a
counterexample to the original Pach's animal problem. On the other hand, it is
beyond the targets of this paper to provide a sequence of additions and
removals transforming $A$ to a cube possibly outside the bounding box. Thus we
provide only a sketch of an idea how to do so without any guarantee of
correctness: Step 1. It seems possible to magnify $A$ to an animal $A'$
obtained by replacing each cube of $A$ with $k^3$ cubes in $A'$ (for
arbitrarily chosen $k > 0$); see also Figure~\ref{f:dilation} in the appendix
for a 2-dimensional analogy. In our case,
$A$ does not contain ``undilatable patterns'' thus the approach sketched in
Nakamura's report~\cite{nakamura10} seems to work in our case. Step 2. Once we
have $A'$ it is essentially a subdivision of $A$. Here we would use that a
sufficiently deep barycentric subdivision of any traingulated/polytopal ball is
shellable; see~\cite[Cor. I.3.10]{adiprasito-benedetti17}. Then it seems
possible to emulate a shelling of an iterated barycentric subdivision via removals of
cubes in $A'$ provided that $k$ is sufficiently large.\\[1mm] 
The approach suggested here is, however, probably far from being optimal
if we want to minimize the space required for a transformation of $A$ to a single cube. It
is, for example, not known to the author whether it is possible to transform $A$ to a cube
within a box obtained by extending one of the dimensions of the bounding box by
1.}

Finally our problem is loosely related to reconfiguration problems of cubical
or crystalline robots~\cite{abel-kominers08arxiv, akitaya+21, feshbach-sung21, kostitsyna+23arxiv}. The general aim of such reconfiguration problems is to
transform one configuration of a robot (formed by modules which correspond to
grid cubes in our language) to another configuration following certain prescribed rules. Stated in this generality, our problem fits into this
framework; however additions or removals of cubes are perhaps less natural
conditions in the setting of reconfigurations of robots.

\section{Sketch of the construction (mostly in dimension 2)}
We start with a sketch of the construction, then we provide details.

\subparagraph{Furch's construction.}
The starting point is Furch's construction. Here the idea is to take a thickening
of a non-trivial knot made by grid cubes inside a box; see
Figure~\ref{f:furch}. More precisely, the knot here is an image of an interval
but together with a curve on the boundary it is a nontrivial knot in the
standard sense. All the cubes forming the knot are removed from the box except
a single one at one of the boundaries. This way we obtain an animal because we can
remove the cubes of the knot one by one while keeping the homeomorphism type.  
In this example, we can remove a few more cubes by keeping animality (for
example those that are at the corners) but at some point we get a non-trivial
animal from which no other cubes can be removed (not proved here---we are only
sketching an idea). 

\begin{figure}
\begin{center}
\includegraphics[page=2]{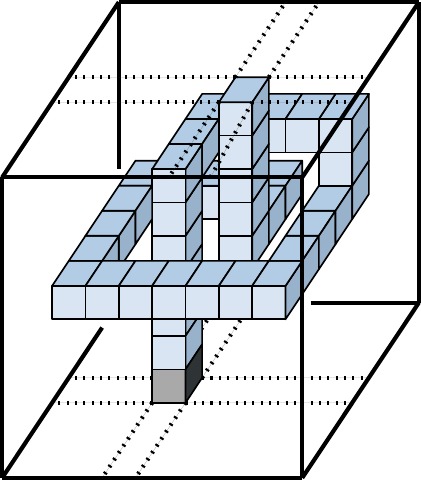}
  \caption{Furch's knotted ball. All displayed cubes are removed from the box
  except the dark one. The picture we provide here is very similar to a picture
  in~\cite{ziegler98}.}
  \label{f:furch}
\end{center}
\end{figure}

\subparagraph{A 2-dimensional example.} 
One of the key ingredients in our construction will be to redesign the knot so
that it fills all ``available'' space. In order to explain this idea (as well
as few other ones), it is beneficial to describe a simplified example in the
plane (but not fully reaching our goals). During the construction we specify
the dimensions of the intermediate objects because these dimensions will be
analogous in the final construction in the 3-space.

First, we consider a square $S$ of dimensions $4\times4$ subdivided into 16 smaller (unit) squares together with a certain (red) piecewise linear curve; see
Figure~\ref{f:first_expansion}, left. The segments of the red curve connect the
centers of pairs of adjacent squares, or possibly centers of the edges of the
squares on the boundary. Next we expand this picture as follows (see
Figure~\ref{f:first_expansion}, right): Each point $p$ in $S$ 
with coordinates which are integer multiples of $1/2$ will correspond to a unit
square $S(p)$ in the expanded picture. A square $S(p)$ in the expanded picture is
red, if $p$ is on the red curve, otherwise it is black. This way we obtain a
square $S'$ of dimensions $9\times9$ where some of the squares are black and other
ones are red. The reader is encouraged to think about this construction as an
analogy of Furch's construction in dimension 2 where the red squares correspond
to the cubes of the knot. (The aim of the introduction of the red curve and the
expansion is to describe the construction efficiently.)

\begin{figure}
\begin{center}
\includegraphics[page=3]{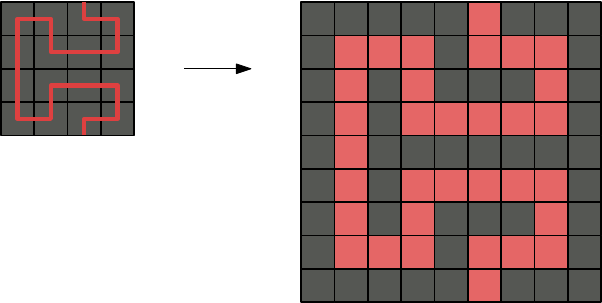}
  \caption{The first expansion of a 2-dimensional example.}
  \label{f:first_expansion}
\end{center}
\end{figure}

Next we extend $S'$ to a rectangle of dimensions $9 \times 14$ as shown in
Figure~\ref{f:second_expansion}, left. The boundary squares are black, while
the remaining squares are white. We again draw a red curve through the red and
white squares as depicted. Inside the black squares, we connect the centers of
two black squares by a segment if they share an edge (essentially constructing
the dual graph). Next we perform an analogous expansion as in the previous
case; see Figure~\ref{f:second_expansion}, right. (For purposes of this sketch, we mostly rely on the picture; we only keep
the squares that correspond to the points on the red curve or black segments.
Because of this, the dimensions of the box after the second expansion are $17
\times 27$.)

\begin{figure}
\begin{center}
\includegraphics[page=4]{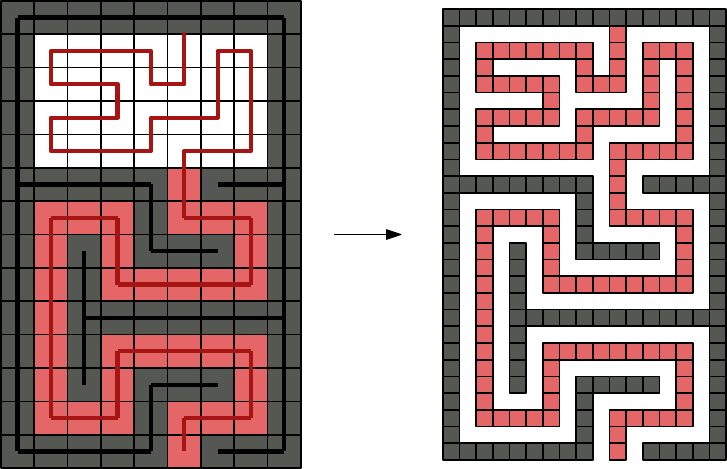}
  \caption{The second expansion of a 2-dimensional example. The squares on both
  sides of the picture should be understood as unit squares. The dimensions of
  the right right picture are $17 \times 27$ but it is
  shrunk due to space constraints.}
  \label{f:second_expansion}
\end{center}
\end{figure}

Finally, we double the construction by taking the mirror copy of the collection
obtained so far, putting this mirror copy from below and connecting the red
squares of the two copies; see Figure~\ref{f:two_copies}, left.

\begin{figure}
\begin{center}
\includegraphics[page=5]{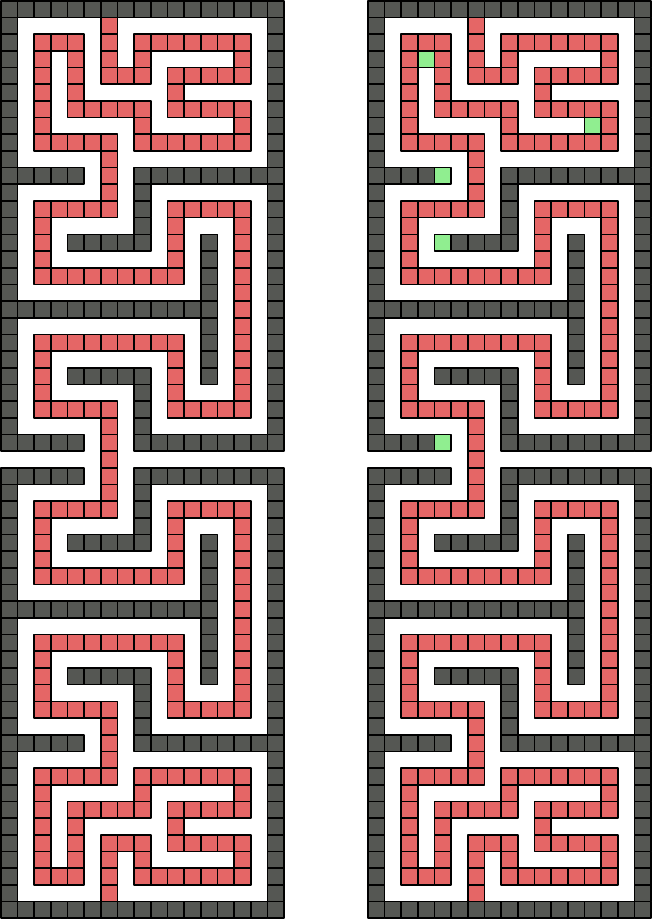}
  \caption{Left: Joining the construction from Figure~\ref{f:second_expansion}
  with its mirror copy. Now the dimensions are $17 \times 55$. Right: After adding or
  removing the squares in green we still have a 2-dimensional animal.}
  \label{f:two_copies}
\end{center}
\end{figure}

Now, it is easy to check that the black and the red squares in the final
construction form a $2$-dimensional animal. In dimension 2, it is easily
possible to remove or add squares within the bounding box so that we preserve
animality. Examples of squares that can be added or removed are depicted in
Figure~\ref{f:two_copies}, right. However, all such squares either correspond to
U-turns of the red curve (before the first expansion or before the second
expansion), or to the spots where the red curve leaves the original square $S$.

The main idea for the proof of Theorem~\ref{t:main} is that it is possible to
perform an analogous construction in dimension 3 so that U-turns are
avoided and the exceptional cases corresponding to the spots where the red
curve leaves the original square $S$ can be set up in a way that no addition or
removal of cubes is possible while preserving animality. This will prove
Theorem~\ref{t:main}.

\section{Full construction}

Now we provide the details for the 3-dimensional construction analogous to the
example in the previous section. 

\subparagraph{Box filling curves.} We start with a construction of
curves in a box that avoid U-turns. More precisely by a \emph{curve
filling a box} we mean a piecewise linear simple curve inside a box $B$ of dimensions $a \times b
\times c$ formed by $abc$ grid cubes, satisfying the following properties: (i)
The curve passes through the center of every grid cube inside the box; (ii) the
curve is composed of segments connecting the centers of pairs of grid cubes
sharing a square (that is, the pairs that are adjacent in the strongest
possible sense); and (iii) the curve does not contain any \emph{U-turn}; that is,
the curve does not pass through four different centers of grid cubes
$m_1, m_2, m_3, m_4$ consecutive in this order along the curve so that $m_2 -
m_1 = -(m_4 - m_3)$; see Figure~\ref{f:u-turn}. (If a curve
contains a U-turn, then it can be easily locally shortened within its isotopy
class making some space empty. For purposes of this paper we do not regard this as honestly filling the box.)

\begin{figure}
\begin{center}
\includegraphics[page=7]{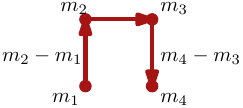}
  \caption{U-turn.}
  \label{f:u-turn}
\end{center}
\end{figure}

Surprisingly, box filling curves can already be found in quite small boxes
(ignoring the trivial case of a box $a \times 1 \times 1$). In our
construction, we will need a specific curve $\gamma_{444}$ in a box $4 \times 4 \times 4$ and
another one $\gamma_{774}$ in a box $7 \times 7 \times 4$. (It should not be a surprise that
these curves were found via a computer search.) These curves are drawn in
Figure~\ref{f:curves} layer by layer. (By a \emph{layer} we mean a collection
of cubes whose centers have the same $z$-coordinate.)

\begin{figure}
\begin{center}
\includegraphics[scale =.95, page=9]{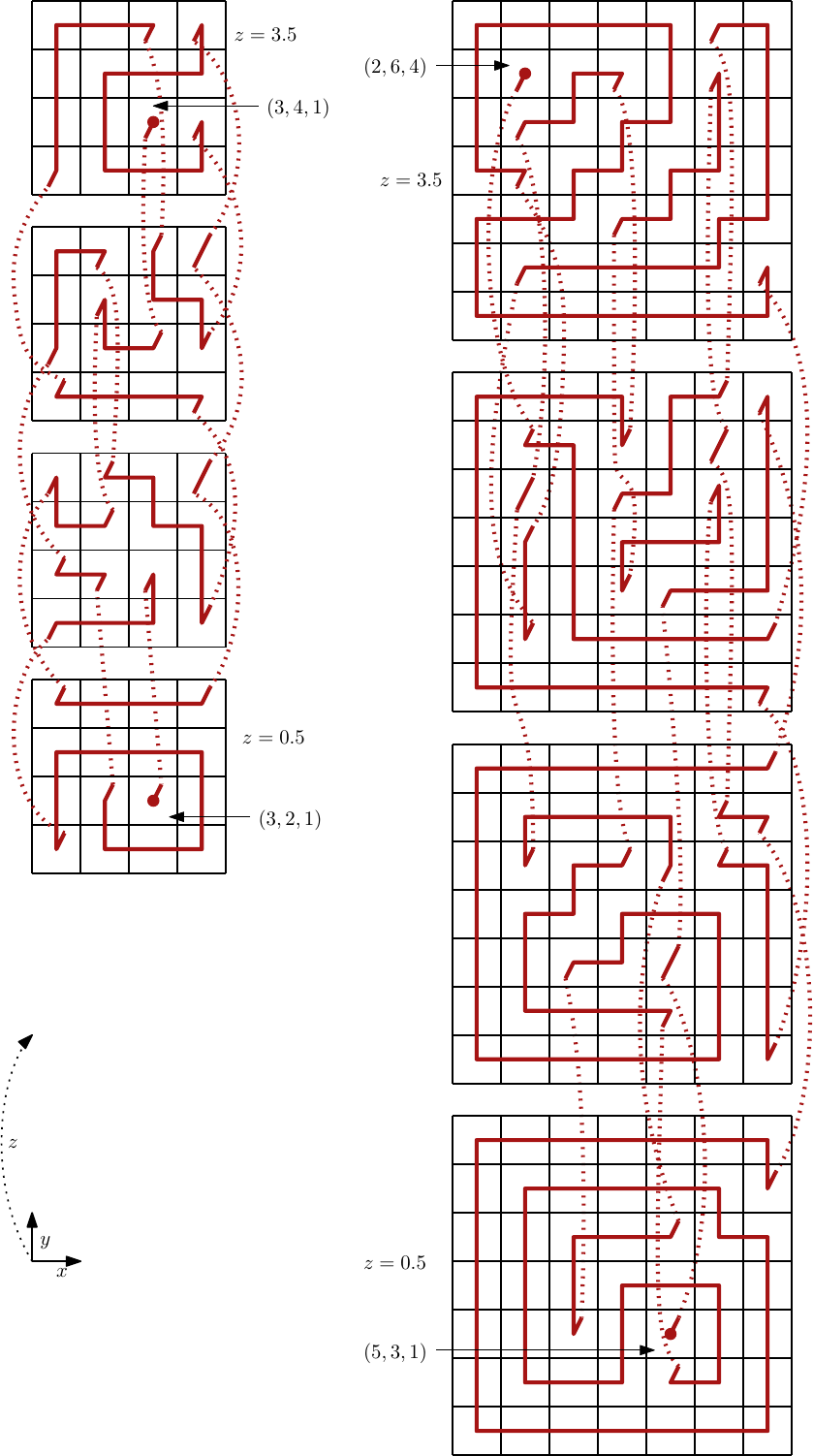}
  \caption{Box filling curves.}
  \label{f:curves}
\end{center}
\end{figure}

Throughout the rest of the construction, a grid cube $[a-1,a] \times [b-1,b]
\times [c-1,c]$ will also be called the $(a, b, c)$ cube. The curve
$\gamma_{444}$ fills the box $[0,4]^3$ and it connects the cube $(3,2,1)$ with
the cube $(3,2,4)$. It
will also be important later on that the second cube along the curve is
$(3,2,2)$ and the last but one cube is $(3,2,3)$. The curve $\gamma_{774}$ fills the
box $[0,7]^2 \times [0,4]$ and it connects the
cube $(5,3,1)$ with the cube $(2,6,4)$; the second cube is $(5,3,2)$ and the
last but one cube is $(2, 6, 3)$. It is routine (but slightly tedious) to
check that Figure~\ref{f:curves} indeed depicts curves (there is no hidden
cycle). It is again routine to check that there are no U-turns. (It is
immediately visible that there is no U-turn in a layer orthogonal to the
$z$-axis. Other U-turns would either consist of a segment in such a layer where
both endpoints continue up or both endpoints continue down---there is no such a
segment---or they would appear at the end of the path in a layer---there is no
such U-turn either but this requires a bit more effort to check because it is
necessary to check how the curve continues in the next layer.)

\subparagraph{The first expansion.} Now we proceed analogously as in the
2-dimensional case. We take the box $B_1 = [0,4]^3$ and we consider $\gamma_{444}$
inside this box. Because $\gamma_{444}$ connects the centers of cubes 
$(3,2,1)$ and $(3,2,4)$ it connects the points $p_1 := (2.5, 1.5, .5)$ and $q_1
:= (2.5, 1.5, 3.5)$. We extend this curve to the boundary by adding segments
$p'_1p_1$ and $q_1q'_1$ where $p'_1 := (2.5, 1.5, 0)$ and $q'_1 := (2.5,
1.5, 4)$. Then we declare the resulting
curve as the \emph{first red curve}. Altogether, the red curve connects the
points $p'_1$ and $q'_1$.

Next we perform the expansion. We consider a mapping $\phi$ from points in
$B_1$ with coordinates divisible by $1/2$ to grid cubes defined so that it maps
a point $p = (a/2, b/2, c/2)$ in $[0,4]^3$ to the grid cube $(a,b,c)$. With a
slight abuse of the notation, we extend this map to arbitrary subsets of $\R^3$: 
If $X \subset \R^3$, then by $\phi(X)$ we mean a grid complex formed by cubes
$(a,b,c)$ where $(a/2, b/2, c/2)$ belongs to $X$. (We also allow $X$ to be a
grid complex, and then we mean that $\phi$ is applied to the polyhedron of this
complex.) We remark that $\phi(B_1)$ consists of cubes filling the
box $B_2 = [-1,8]^3$. (Note that the lexicographically smallest cube in 
this box is $(0,0,0)$.)
We color a cube in $[-1,8]^3$ \emph{red} if the corresponding point belongs to
the red curve, otherwise it is colored \emph{black}. Note that the dual graph
of the red cubes is a path, denoted $\pi$, connecting the centers of cubes
$\phi(p'_1) = (5,3,0)$ and $\phi(q'_1) = (5,3,8)$. (Here we again use a slight
abuse of the notation identifying a cube with its coordinates.)
Here by the \emph{dual graph} of a collection of cubes
we mean the graph defined so that its vertices are centers of the cubes in the
collection and the edges, realized as straight line segments, connect the
centers of those pairs of cubes which share a square.

\subparagraph{The second expansion.} Now we extend $B_2$ to the box
$B'_2 = [-1,8]^2 \times [-1,13]$. Inside the ``subbox'' $[-1,8]^2 \times [8,13]$ we
declare a cube \emph{black} if it is on the boundary of $[-1,8]^2 \times
[-1,13]$; that is the (new) black cubes are of the form $(a,b, 13)$, $(a, 0, c)$, $(a, 8, c)$, $(0, b, c)$ or $(8, b, c)$ for $a, b \in \{0, \dots, 8\}$ and $c
\in \{9, \cdots, 13\}$. All remaining cubes of the subbox are white---these are
cubes inside \emph{white box} $W := [0,7]^2 \times [8,12]$. Note that the dimensions
of the white box are $7 \times 7 \times 4$ suitable for putting $\gamma_{774}$
inside it.

We obtain the \emph{second red curve} as follows. First, we take the path $\pi$
considered as a curve connecting the centers of the cubes $\phi(p'_1)= (5,3,0)$ and
$\phi(q'_1) = (5,3,8)$. Next we attach the segment connecting the centers of the cubes
$(5,3,8)$ and $(5,3,9)$. Next we attach (shifted) $\gamma_{774}$ inside the
white box. Because the white box is $[0,7]^2 \times [8,12]$ instead of $[0,7]^2
\times [0,4]$, this means that this $\gamma_{774}$ connects the centers of
cubes $(5,3,9)$ and $(2, 6, 12)$ (instead of $(5,3,1)$ and $(2,6,4)$). We
denote the endpoints of $\gamma_{774}$ as $p_2 = (4.5, 2.5, 8.5)$ and $q_2 =
(1.5, 5.5, 11.5)$.
Finally, we attach the curve to a black cube on the boundary of $[-1,8]^2 \times
[-1,13]$ by adding the segment $q_2q'_2$ where $q'_2 = (1.5, 5.5, 12)$.
This finishes the
construction of the second red curve. We recapitulate that it connects the points
$(4.5, 2.5, -.5)$ and $q'_2$.

Next we obtain a \emph{black dual complex} as follows (see Figure~\ref{f:black_dual} for a simplified example): First, we take the dual
graph of the collection of black cubes. Next, whenever four black distinct
cubes share an edge, we add a square to the complex which is the convex hull of
the centers of the four cubes. (Seemingly, we should also add a cube to the
black dual complex corresponding to eight black cubes sharing a vertex but such eight cubes do
not appear in our construction.)

\begin{figure}
\begin{center}
\includegraphics[page=24]{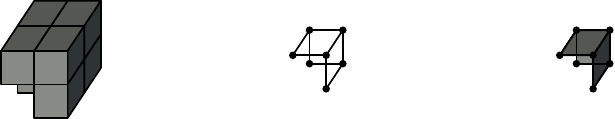}
  \caption{A simplified example of a construction of the black dual complex.
  Left: A collection of seven cubes for which we construct an analogy of the
  black dual complex. Middle: The dual graph of these cubes. Right: The
  resulting complex for these seven cubes.}
  \label{f:black_dual}
\end{center}
\end{figure}

Now we perform the second expansion. We observe that the cubes of
$\phi(B'_2)$ fill the box $B_3 = [-2, 15]^2 \times [-2,25]$.
Now we color the cubes in $B_3$, black, red or white as follows. 
(The coloring should be understood relative to the chosen box. For example, it
is completely plausible if an $(a,b,c)$ cube is red in $B'_2$ and black in
$B_3$. In other words, it is convenient to use the integer coordinates for both
$B'_2$ and $B_3$ but otherwise they should be considered as independent.)
The cube $Q = (a, b, c)$ in $B_3$ is \emph{black} if the corresponding point
$\phi^{-1}(Q)$ in $B'_2$ is in the black dual complex, it is \emph{red}, if the corresponding
point is on the second red curve, and it is \emph{white} otherwise.

%
%

\subparagraph{Bends and straight segments after the second expansion.}
Before we continue with construction, we describe in more detail how the bends
of the first red curve, or the bends of $\gamma_{774}$ as a part of the second
red curve affect the colors of cubes in $B_3$. This will be useful later on in the
proof of Theorem~\ref{t:main}.

Consider a grid cube $Q_1$ in $B_1$. The first red curve may either pass through
$Q_1$ as a straight segment connecting the centers of opposite squares of
$Q_1$, or it may form a bend connecting the centers of non-opposite squares
through the center of $Q_1$; see Figure~\ref{f:magnify_bends}, left.

\begin{figure}
  \begin{center}
    \includegraphics[page=11]{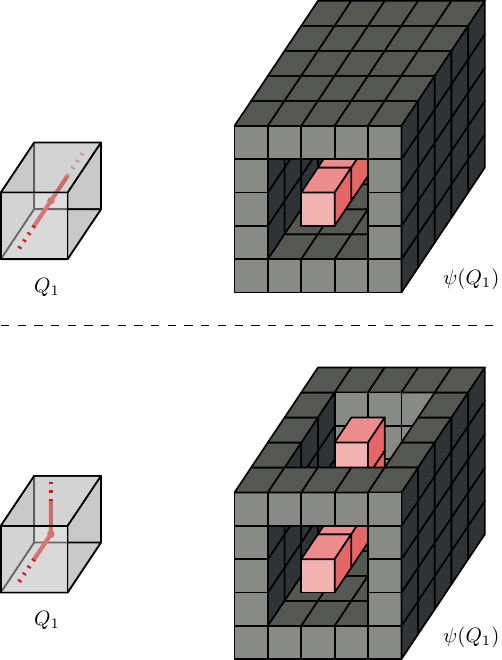}
    \caption{The two expansions of grid cubes in $B_1$. The white cubes are not depicted.}
    \label{f:magnify_bends}
  \end{center}
\end{figure}

Then it follows from the construction that the way how the first red curve passes
through $Q_1$ uniquely determines the colors of $5^3$ cubes in $\psi(Q_1)$
where $\psi(Q_1)$ is obtained by applying $\phi$ to the interior of (the
polyhedron of) $\phi(Q_1)$; see
Figure~\ref{f:magnify_bends}, right. We emphasize that this can be used even for cubes $(3, 2, 1)$
and $(3, 2, 4)$ where $\gamma_{444}$ starts and ends because the red curve is
extended in these cubes. (And it does not make a bend due to the properties of
$\gamma_{444}$.) We also remark that the boxes $\psi(Q_1)$ in general
overlap. For example, if $Q_1$ and $Q'_1$ share a square, then $\psi(Q_1)$
and $\psi(Q'_1)$ share $25$ cubes.

A very analogous description of the expansion can be provided for the cubes in
the white box $W$ of $B'_2$: This time we consider a cube
$Q_2$ in the white box $W$. There are again two ways how the second red curve may
pass through $Q_2$; see Figure~\ref{f:magnify_wb}, left. This determines the
colors (red or white) of $3^3$ cubes of $\phi(Q_2)$; see Figure~\ref{f:magnify_wb},
right. This applies also to the cubes where $\gamma_{774}$
starts or ends because the second red curve is extended beyond these endpoints.

\begin{figure}
  \begin{center}
    \includegraphics[page=12]{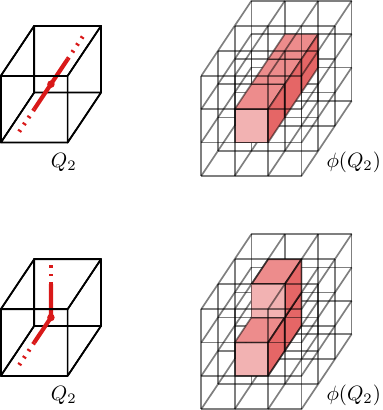}
    \caption{The second expansions of grid cubes in the white box of $B_3$. The
    white cubes are drawn as transparent.}
    \label{f:magnify_wb}
  \end{center}
\end{figure}

\subparagraph{Doubling.}

Finally, we extend $B_3$ to the box $B'_3 = [-2, 15]^2 \times [-30,25]$---this
is our final bounding box. Again we color cubes in $B'_3$ black, red or white.
For cubes that belong to both $B_3$ as well as $B'_3$ we keep the colors as
they are in $B_3$. Inside $[-2, 15]^2 \times [-30,-3]$ we take a mirror copy of
$B_3$ so that the colors are symmetric along the \emph{mirror plane} $\mu$ given by the
equation $z = -2.5$. It remains to color the cubes $(a,b, -2)$ for $a, b \in
\{-1, \dots, 15\}$ (these cubes have their centers on the aforementioned
plane).  Among these cubes, we color the cube $(9,5,-2)$ red, all other cubes
are white. (Note that the red cube $(9,5,-2)$ shares a square with exactly two
other red cubes $(9,5,-1)$ and $(9, 5, -3)$.

This finishes the construction. The desired animal $A$ from
Theorem~\ref{t:main} is the grid complex consisting of the black and the red cubes in
$B'_3$. It remains to check that $A$ is indeed an animal and that the result of
removing any cube or adding any cube within the bounding box is not an animal.
This is done in the next section.

\section{Correctness of the construction.}
In this section, we finish the proof of Theorem~\ref{t:main}.

\subsection{$A$ is an animal}
   
In this subsection we work in PL (piecewise-linear)
category~\cite{rourke-sanderson72}. In particular, a ball means a PL ball and
a manifold means a PL manifold.\footnote{We work in dimension $3$. It is well
known that the topological category and PL category coincide in dimension 3.
However, we still need a lemma stated in the PL world. (And we will not be
really using that those two categories coincide.)}
We will repeatedly need a special case of Lemma~3.25
from~\cite{rourke-sanderson72} which can be phrased in dimension 3 as follows:
\begin{lemma}
  \label{l:shelling}
Let $M_1$ and $N$ be $3$-manifolds meeting in a disk (a. k. a. 2-ball) on the
  boundary of both $M_1$ and $N$. Assume that $N$ is a $3$-ball. Then $M := M_1
  \cup N$ is a $3$-ball if and only if $M_1$ is a $3$-ball.
\end{lemma}

Let $R$ be the union of red cubes in $A$, let $K_+$ be the union of black cubes
in $A$ above the mirror plane $\mu$ and $K_-$ be the union of black cubes in $A$
below $\mu$. Note that $K_+$ and $K_-$ are disjoint (because
they are separated by $\mu$ which does not intersect any black cube) 
while both of them meet $R$ in a square. Indeed, $K_+$ meets $R$ in the square shared
by the cubes $(3, 11, 24)$ and $(3,11, 25)$. (Recall that $q'_2 = (1.5, 5.5, 12)$ is
one of the endpoints of the second red path, thus the center of the cube
$\phi(q'_2) = (3, 11, 24)$ is the endpoint of the path which is
the dual graph of $R$.) It is also straightforward to check that there are no
other places where $R$ and $K_+$ would intersect by following the second red curve inside the
white box and the (first) red curve inside $B_1$ and checking the result of
expansion(s) with an aid of Figures~\ref{f:magnify_bends} and~\ref{f:magnify_wb}. (It may
be also useful to compare this with the 2-dimensional analogue; see
Figure~\ref{f:two_copies}.) The argument for $K_-$ and $R$ is the same, mirror
symmetric.

Thus, in order to show that $A$ is an animal, it
is sufficient to show that $R$, $K_+$ and $K_-$ are 3-balls (using
Lemma~\ref{l:shelling} twice, first for $R$ and $K^+$, then for $R \cup K^+$
and $K^-$).

It is easy to check that $R$ is a 3-ball because the dual graph of the cubes in
$R$ is a path. Thus, $R$ can be obtained from a single cube by adding cubes one
by one meeting the previous collection in a square. It follows from a repeated
application of Lemma~\ref{l:shelling} that $R$ is a 3-ball.

Now we check that $K^+$ is a ball (the case of $K^-$ is symmetric). The idea is
that we obtain $K^+$ from the box $B_3$ (which is a 3-ball) by drilling a hole into it.
It is hard to provide a $3$-dimensional figure for this approach but it is easy
to display an analogy in $2$-dimensional setting; see Figure~\ref{f:animal_check}.
In detail, we follow the first red path cube by cube in $B_1$ and in each case
when we meet a cube $Q_1$ of $B_1$, we consider the box $\psi(Q_1)$; see
Figure~\ref{f:magnify_bends}. From this box, we remove a ``subbox'' of dimensions $3$,
$3$ and $4$. This subbox consists of 27 (red or white) cubes  not touching the
boundary of $\psi(Q_1)$ and 9 cubes (1 red, eight white) corresponding to the side (square) 
from which we have entered $Q_1$. Due to Lemma~\ref{l:shelling}, each such step
preserves the fact that we have a $3$-ball. At the very last cube of $B_1$,
we are in the `straight' case (due to the properties of
$\gamma_{444}$) and we remove all $3 \cdot 3 \cdot 5$ red or white cubes inside
the affected $5^3$ cubes. At this moment, the red and white cubes are inside the box
$\phi(W) = [-1, 14]^2 \times [16,24]$. Now, we remove $\phi(W)$ in a single step obtaining
$K^+$. We again know that we got a $3$-ball by Lemma~\ref{l:shelling}. This
finishes the proof that $A$ is an animal.

\begin{figure}
  \begin{center}
    \includegraphics[page=17]{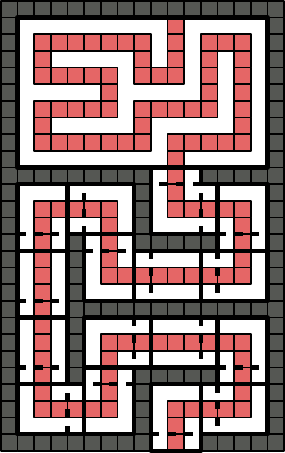}
    \caption{Checking that $A$ is an animal. The $3 \times 4$ boxes correspond
    to the $3 \times 3 \times 4$ boxes in the $3$-dimensional setting. The $3
    \times 3$ squares inside them correspond to the $3 \times 3 \times 3$ boxes
    in dimension~$3$. The final bend in the $2$-dimensional picture does not
    appear in the dimension $3$.}
    \label{f:animal_check}
  \end{center}
\end{figure}

\subsection{Additions and removals}

Now we prove that no cube within the bounding box of $A$ can be added to $A$ or
removed from $A$ while keeping that we have an animal. We will need the
following lemma characterizing removable or addable cubes, which must be a folklore.

\begin{lemma}
\label{l:remove}
  Let $A$ be an animal and $Q$ be a grid cube. Let $A'$ be the grid complex
  obtained from $A$ by removing $Q$ if $Q$ is in $A$, or by adding $Q$ if $Q$ is
  not in $A$; and let as assume that $A'$ is an animal as well. Let $N(Q)$ be the grid complex formed by the 26 grid 
  cubes different from $Q$ which intersect $Q$. Finally, let $A^+ := A \cap
  N(Q)$ and $A^-$ be the subcomplex of $N(Q)$ formed by those cubes which do
  not belong to $A$. Then $Q$ meets both $A^+$ and $A^-$ in a disk.
\end{lemma}

\begin{proof}
  
  First of all, in the proof we can assume that $A$ does not contain $Q$ and thus
$A'$ is obtained by adding $Q$ to $A$. If this is not the case, we can just
  swap $A$ and $A'$ in the proof (the definitions of $A^+$ and $A^-$ are
  resistant to this swap).

Now, we set $Q^+ := Q \cap A^+ = Q \cap A$ and $Q^- := Q \cap A^-$.
  Next we observe that both $Q^+$ and $Q^-$ are pure $2$-dimensional; that is,
  every vertex and edge is contained in a square. See Figure~\ref{f:singular}
  when following the proof. For vertices, it is sufficient to show that there
  are no isolated vertices because a vertex in an edge belongs to a square if
  the edge belongs to a square.  If $Q^+$ contains an isolated vertex $v$ (that
  is, not contained in any edge of $Q^+$), then $v$ is a singular point of $A'$
  contradicting that $A'$ is an animal. If $Q^-$ contains an isolated vertex
  $v$, then $v$ is a singular point of $A$. If $Q^+$ contains an edge $e$ which
  is not contained in any square of $Q^+$, then any interior point of $e$ is a
  singular point of $A'$ and if $Q^-$ contains an edge $e$ which is not
  contained in any square of $Q^-$, then any interior point of $e$ is a
  singular point of $A$.

  \begin{figure}
    \begin{center}
      \includegraphics[page=10]{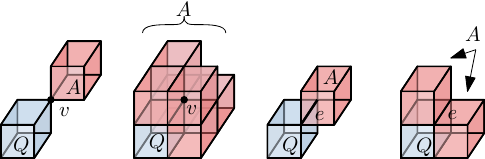}
      \caption{Cases when the singular points appear. Only the cubes that
      contain $v$ or $e$ are displayed.}
      \label{f:singular}
    \end{center}
  \end{figure}

At this moment we know that $Q^+$ and $Q^-$ are formed by collections of
  squares and the edges and vertices in these squares. Note that a square of
  $Q$ belongs to $Q^+$ if and only if it does not belong to $Q^-$. Thus it is
  sufficient to rule out the following cases: One of $Q^+$ or $Q^-$ contains no
  square or one of $Q^+$ or $Q^-$ contains exactly two opposite squares. (All
  other cases yield disks.) We immediately rule out the case $Q^+ = \emptyset$
  as $A'$ would be disconnected in this case. In all other cases, we utilize the
  Mayer-Vietoris exact sequence for reduced homology~\cite[Chap.~2.2]{hatcher02} using that
  $A' = A \cup Q$ and $Q^+ = A \cap Q$: 
  \[ \cdots \to \tilde H_{n+1}(A') \to \tilde H_{n}(Q^+) \to \tilde H_n(A) \oplus
  \tilde H_n(Q) \to \tilde H_n(A') \to
  \cdots \]

  From the fact that $A$ and $A'$ are animals and $Q$ is a cube, we get that
  $\tilde H_{n+1}(A'), H_n(A), H_n(Q)$ and $H_n(A')$ are all trivial. From
  exactness it follows that $H_{n}(Q^+)$ is trivial as well. But this rules out
  all previous cases. (If $Q^+$ consists of two opposite squares, then
  $\tilde H_0(Q^+) \cong \Z$; if $Q^+$ consists of four squares missing two
  opposite squares, then $\tilde H_1(Q^+) \cong \Z$; and if $Q^+$ consists of
  all six squares, then $\tilde H_2(Q^+) \cong \Z$.)
\end{proof}

We finish the proof of Theorem~~\ref{t:main} by distinguishing whether we
remove a red or a black cube or whether we add a with cube within the bounding
box. For clarity of the structure of the proof, it is useful to state this in
three separate lemmas.

\begin{lemma}
\label{l:red}
  Removing a red cube $Q$ from $A$ yields a grid complex $A'$ which is not an animal.
\end{lemma}

\begin{lemma}
\label{l:white}
  Adding a white cube $Q$ in the bounding box $B'_3$ to $A$ yields a grid complex $A'$ which is not an animal.
\end{lemma}

\begin{lemma}
\label{l:black}
  Removing a black cube $Q$ from $A$ yields a grid complex $A'$ which is not an animal.
\end{lemma}

\subparagraph{Joint preliminaries for the proofs of all three lemmas.}
In all three proofs, we use the notation $Q^+$ and $Q^-$ from the proof of
Lemma~\ref{l:remove}. That is, $Q^+ := Q \cap A^+$ and $Q^- := Q \cap A^-$
where $A^+$ and $A^-$ are as in the statement of Lemma~\ref{l:remove}. (In
other words $Q^+$ is formed by the intersection of $Q$ with the black and red
cubes while $Q^-$ is formed by the intersection of $Q$ with the white cubes.
(We consider the cubes outside the bounding box also as white.)

In all three proofs, {\bf we assume for contradiction} that $A'$ is an animal.
Using Lemma~\ref{l:remove} we deduce that both $Q^+$ and $Q^-$ are disks. We
will consider different cases of how $Q$ is positioned in $A$. Each such case
will yield a contradiction with this disk property.

We also emphasize that in all three cases, it is sufficient to consider the
cubes which are either in $B_3$ or meet the mirror plane $\mu$; other cases are
mirror symmetric.

\begin{proof}[Proof of Lemma~\ref{l:red}.]
The dual graph of the
  collection of the red cubes $R$ is a path. If $Q$ corresponds to a vertex of
  degree two of this path, then either there is a bend at $Q$ in which case,
  $Q^-$ contains an edge not in square, thus $Q^-$ is not a disk; or there is
  no bend and $Q^+$ consists of two opposite squares---see Figure~\ref{f:redQ}
  covering both cases. Both cases show that $A'$ is not an animal. The only
  remaining case of $Q$ red is that $Q$ is the cube $(3, 11, 24)$. In this
  case, $Q^+$ again consists of two opposite squares as $(3, 11, 25)$ is black,
  $(3,11, 23)$ is red while all eight cubes $(a,b, 24)$ intersecting $Q$ but
  different from $Q$ are white.
\end{proof}

  \begin{figure}
    \begin{center}
      \includegraphics[page=13]{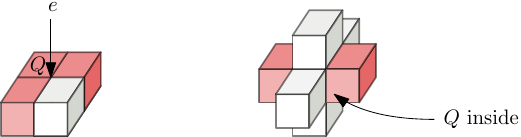}
      \caption{A neighborhood of a red cube $Q$. (Only some of the cubes for which we
      can determine the color are displayed.)}
      \label{f:redQ}
    \end{center}
\end{figure}

\begin{proof}[Proof of Lemma~\ref{l:white}.]

 We first resolve the easy cases:
  If $Q$ intersects the mirror plane $\mu$, then $Q^+$ meets both $K^+$ and
  $K^-$, thus it immediately follows that $Q^+$ has at least two components.
  Thus it is not a disk. Similarly, if $Q$ belongs to the box $[-2,15]^3$
  (which is the union of $\psi(Q_1)$ for cubes $Q_1$ in $B_1$), then
  it follows from the construction that $Q$ meets both $R$ and $K^+$ (see also
  Figure~\ref{f:magnify_bends}) but it does not meet $R \cap K^+$. It again
  follows that $Q^+$ has at least two components.

 It remains to consider the case that $Q$ belongs to the box $\psi(W) = [-1,14]^2
  \times [16, 24]$. (Note that the cubes in $[-2,15]^2 \times [16,25]$ but not
  in $\psi(W)$ are all black.) 
  
  We observe that the cubes $(a, b, c)$ in $\psi(W)$ for which $a$, $b$
  and $c$ is odd are red. These cubes are central cubes in boxes $\psi(Q_2)$
  where $Q_2$ is in (the white box) $W$; see Figure~\ref{f:magnify_wb}. We will
  call these cubes \emph{central}.

Now we separately consider the case when $Q$ is on the boundary of $\psi(W)$.
  In this case $Q$ has to meet both a red cube and a black cube. If $Q$ does
  not meet the intersection of $R$ and $K^+$, then $Q^+$ has two components
  which is the required contradiction. If $Q$ meets the intersection then, $Q$
  is one of the eight white cubes meeting the red cube $\phi(q'_2) = (3,11,24)$
  and in the same layer as this red cube. In this the neighborhood of any such
  $Q$ could be fully reconstructed from Figure~\ref{f:curves}; however we
  provide an argument independent of the exact way how $\gamma_{774}$ passes
  through the white box (except that we know that there is no bend at $q_2$);
  see Figure~\ref{f:white_end_cube}. If $Q$ shares an edge with $\phi(q'_2)$,
  then this edge is not in any square of $Q^+$, thus $Q^+$ is not a disk. If
  $Q$ shares a square with $\phi(q'_2)$, then $Q^+$ shares an edge with one of
  the four red (central) cubes depicted in layer 23 of Figure~\ref{f:white_end_cube}
  different from $\phi(q_2)$. This edge is not in any square of $Q^+$, thus
  $Q^+$ is not a disk.

\begin{figure}
    \begin{center}
      \includegraphics[page=18]{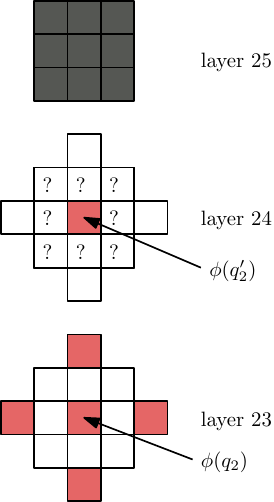}
      \caption{A neighborhood of a white cube $Q$ which meets both $R$ and
      $K^+$; $Q$ is one of the eight
      cubes marked with `?'. (Only some of the cubes are displayed.)}
      \label{f:white_end_cube}
    \end{center}
\end{figure}

  It remains to consider the case when $Q$ is $\phi(W)$ but not on the boundary
  of it. In this case, $Q$ meets no black cube. In this case, $Q$ may meet
  central cubes in vertices (if all coordinates of $Q$ are even), in edges (if
  one of the coordinates of $Q$ is odd), or in squares (if two of the
  coordinates of $Q$ are odd). In addition, a we remark that a cube with at
  most one coordinate odd is necessarily white (this easily follows from an
  inspection of Figure~\ref{f:magnify_wb}). 

  If $Q$ meets central cubes in vertices, then all six cubes neighboring $Q$
  (that is sharing a square) must be white, thus $Q^-$ is not a disk. If $Q$ meets the central
  cubes in squares, then two opposite cubes neighboring $Q$ are red while the
  remaining four cubes are white. This means that $Q^-$ is not a disk. Finally
  assume that $Q$ meets the central cubes in edges; see
  Figure~\ref{f:white_u_turn}. Consider the four cubes neighboring $Q$ which
  meet central cubes in squares. These cubes may be white or red, the remaining
  two neighbors of $Q$ are white. All four cubes cannot be red. If three of the
  four cubes are red, then this configuration comes from a U-turn on
  $\gamma_{774}$ which is a contradiction that $\gamma_{774}$ does not form
  U-turns. Finally, if two or less of the cubes are red, then either there is
  an edge in $Q^+$ shared with one of the central cubes which is not in any
  square of $Q^+$, or exactly two opposite neighbors of $Q$ are red. In each of
  these cases $Q^+$ cannot be a disk.
\end{proof}

\begin{figure}
    \begin{center}
      \includegraphics[page=19]{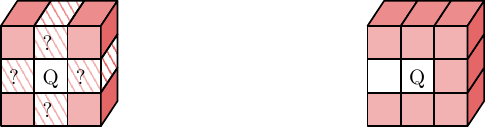}
      \caption{The cube $Q$ meeting the central cubes in edges and the U-turn.}
      \label{f:white_u_turn}
    \end{center}
\end{figure}

\begin{proof}[Proof of Lemma~\ref{l:black}]
  We start our considerations with the case that 
  $Q$ is on the boundary of $B_3$, then it intersects a white cube
outside $B_3$ as well as it intersects a white cube inside $B_3$. If it does
not intersect a white cube on the boundary of $B_3$, then we get that $Q^-$ has
at least two components. If it intersects a white cube on the boundary of
$B_3$, then $Q$ is one of the 16 cubes emphasized in brown in
Figure~\ref{f:brown}, left. (The left picture shows a part of the lowest level
of $B_3$.) Each such cube intersects one of the emphasized white cubes in
Figure~\ref{f:brown}, right in a vertex or edge that cannot belong to a square
of $Q^-$. Thus $Q^-$ cannot be a disk. (Note that we crucially use that
the first red curve does not have a bend in cube $(3,2,1)$, considered in
$B_1$.) 

\begin{figure}
    \begin{center}
      \includegraphics[page=15]{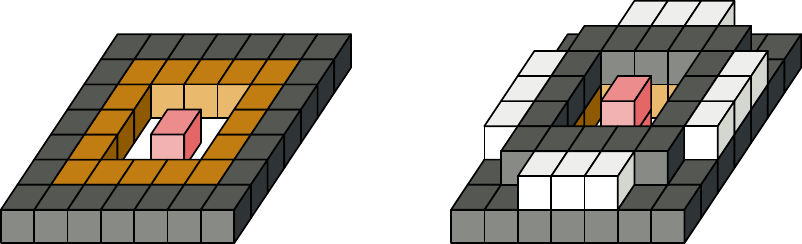}
      \caption{The cubes on the boundary of $B_3$ which intersect a white cube
      on the boundary of $B_3$.}
      \label{f:brown}
    \end{center}
\end{figure}

  Similarly, if $Q$ is in layer 15, which is on the boundary of $[-2,15]^3$,
  then intersects a white cube in $[-2,15]^3$ as well as a white cube outside
  $[-2,15]^3$. If $Q$ does not intersect a white cube in layer 15, then we get
  that $Q^-$ has at least two components. If $Q$ intersects a white in layer
  15, then we are essentially in  the same case as before, only the
  Figure~\ref{f:brown} should be mirrored through a plane orthogonal to the
  $z$-axis. (Here we use that the first red curve does not have a bend in cube
  $(3,2,4)$, considered in $B_1$.) We resolve this case in the same way as the
  analogous case before.

  Finally, there are no black cubes in $\phi(W)$ thus we may assume that $Q$ is
  in $[-2,15]^3$ but not on the boundary. That is, $Q$ is in the box $[-1,
  14]^3$. Here the consideration is essentially the same as for the last cases
  of the previous lemma (only the picture is somewhat rescaled). But we of
  course provide full detail.

  By computing the coordinates in the first and the second expansion we observe
  that for any cube $Q_1$ in $B_1$, the central red cube of $\psi(Q_1)$ has
  coordinates of form $(4i + 1, 4j + 1, 4k + 1)$ where $i, j$ and $k$ are
  integers. We say that a coordinate of a cube in $[-2,15]^3$ is
  \emph{blackish} if it gives remainder $-1$ modulo $4$. 
  The aforementioned observation implies (by checking Figure~\ref{f:magnify_bends}) that the
  cubes where at least two coordinates are blackish are black. If one
  coordinate is blackish, then the cube can be of any color, and if none of the 
  coordinates is blackish, then the cube is red or white.

  Now if all coordinates of $Q$ are blackish, then the cubes neighboring with
  $Q$ (that is, sharing a square with $Q$) are all black, thus $Q^+$ cannot be
  a disk. If exactly one coordinate of $Q$ is blackish, then $Q$ has exactly
  two white opposite white neighbors (with no coordinate blackish) and all
  remaining four neighbors are black. (In Figure~\ref{f:magnify_bends} this
  corresponds to a setting that $Q$ is on some side of $\psi(Q_1)$ for some
  $Q_1$ without a hole formed by the white and red cubes.)

  The final case is to consider the setting when exactly two coordinates of $Q$
  are blackish. Without loss of generality we assume that the blackish
  coordinates are $x$ and $z$. That is, there are integers $i, j, k$ and $r \in
  \{0,1,2\}$ such that $Q$ is a $(4i -1, 4j + r, 4k-1)$ cube; see
  Figure~\ref{f:black_u_turn}, left. $Q$ is one of the three black cubes. By
  $B(Q)$ we denote the $1 \times 3 \times 1$ box formed by these three cubes (independently of the
  position of $Q$ in this box).
  Then all twelve cubes $(4i - 1 \pm 1, 4j + r', 4k - 1 \pm 1)$ for $r' \in
  \{0, 1, 2\}$ are white; see Figure~\ref{f:black_u_turn}, left again.
  Next we consider the four cubes $(4i-1 \pm 2, 4j + 1, 4k -1 \pm 2)$, these
  cubes have to be red as they are central cubes of $\phi(Q_2)$ for some cubes
  $Q_2$; see Figure~\ref{f:black_u_turn}, middle. We also consider four $1 \times 3
  \times 1$ boxes $\hat B_1, \dots \hat B_4$ sharing with $B(Q)$ rectangles of 
  dimensions $1$ and $3$. The cubes in each of these four boxes $\hat B_i$ are
  either all white or all black. Similarly as in the proof of
  Lemma~\ref{l:white} we distinguish how many of these boxes are white.

\begin{figure}
    \begin{center}
      \includegraphics[page=20]{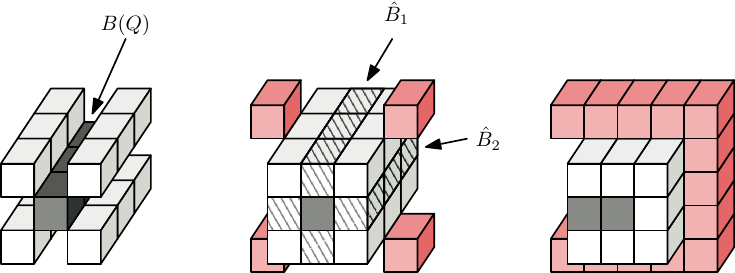}
      \caption{Neighborhood of $Q$ in the last case.}
      \label{f:black_u_turn}
    \end{center}
\end{figure}
 
 All four boxes cannot be white (this would correspond to a loop on
 $\gamma_{444}$. If three of the boxes are white, this corresponds to a U-turn
 on $\gamma_{444}$; see Figure~\ref{f:black_u_turn}, right. If two opposite boxes
 among $\hat B_1, \dots \hat B_4$ are white then $Q$ neighbors two opposite
 white cubes while the remaining neighbors of $Q$ are black. Thus $Q^+$ is an 
 annulus and not a disk. In all remaining cases $Q$ contains an edge share with
 one of the twelve white cube in Figure~\ref{f:black_u_turn}. This edge belongs
 to $Q^-$ but it does not belong to any square of $Q^-$. Thus $Q^-$ is not a
 disk, a contradiction. We have contradicted all cases which finishes the proof
 of the lemma and therefore it finishes the proof of Theorem~\ref{t:main} as
 well.
\end{proof}

\bibliography{animals}

\appendix
\section{Comments on Nakamura's proof.}

Here we provide more details why the author of this paper considers Nakamura's
proof from~\cite{nakamura10} incomplete. Some of the key definitions or steps
in~\cite{nakamura10} leave some freedom in interpretation. It might be possible
that some of the ideas could be reused and the proof could be significantly
reworked into a complete proof. (But the author of this manuscript doubts
whether this is really possible.)

It is not feasible to rewrite and explain all the contents of~\cite{nakamura10}
here in order to point out the gaps. However, we at least try to sketch the
main technique (magnification via dilations). Given an animal, the cubes in the
animal are considered black whereas the cubes outside the animal are considered
white. During magnification and dilations, they are treated in (almost) the
same way. A cube $Q$ with coordinates $(x,y,z)$ can be $x$-dilated to position
$(x',y,z)$ with $x' > x$ so that we gradually color the cubes $(x+1, y, z),
\dots, (x',y,z)$ with the same color as $Q$ provided that we have an animal in
each intermediate step. (If all the cubes are already colored with a correct
color, then we do nothing.) Successive dilations may magnify the animal.
Figure~\ref{f:dilation} displays a magnification
of a 2-dimensional animal first by $x$-dilations and then by (analogous)
$y$-dilations. All the colorful squares should be treated as black but the
colors help to distinguish individual dilations. (Note that the dilations of
the white square and the grey square in fact do nothing.) In general,
magnifications via dilations work well in dimension 2 but there may be
``undilatable cubes'' in dimension 3; see examples in~\cite{nakamura10}.

\begin{figure}        
\begin{center}
  \includegraphics[page=1]{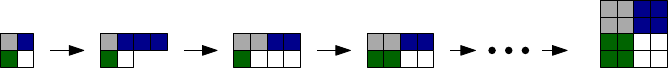}
\end{center}
\caption{Dilation of a $2$-dimensional animal consisting of four squares.}
  \label{f:dilation}
\end{figure}

  Nakamura's proof proceeds in two main steps. The first step is to perform
  dilation in some well defined order but also subject to some side rules. 
  The aim of the first step is to remove ``undilatable cubes''. The second step
  is to magnify the animal coming from the first step and subsequently
  to reduce the magnified animal to a single cube. The second step refers to
  another unpublished technical report~\cite{nakamura-morita-imai06}. We will
  comment mainly on the first step (not judging~\cite{nakamura-morita-imai06}).

Now we provide specific comments (this is probably not a complete list of
problems in~\cite{nakamura10}):

\begin{enumerate}
  \item One of our main objections is that the animal is not \emph{static} but
    \emph{dynamic}. It evolves during the dilations and also the notions such
    as ``undilatable cube'' or ``undilatable pattern'' evolve as well. This
    evolution does not seem to be carefully addressed in~\cite{nakamura10}. A
    specific example is a definition of \emph{subordination} below Claim 4.
    This defines a certain partial order which dynamically depends on the
    intermediate dilations.  However, subordination seems to be used in the
    proof of Claim~5 in a static way. In a proof of Claim~5, the aim is to
    transfer a certain ``undilatable cube'' $p$ into a ``dilatable'' one. If
    this fails because of subordination another cube $h$ (higher in the order) 
    is taken instead. However, $h$ may possibly appear only during the
    dilations. Thus it is not clear whether $h$ is ``undilatable'' in the
    beginning of the process yielding a required simplification. This does not
    seem treated carefully in the manuscript.
  \item Another main objection is that some of the proofs are incomplete.
    Claim~3 in~\cite{nakamura10} states that there is no ``non-applicable''
    outside certain region $R^*$ while the animal evolves during the dilations.
    The proof disusses only one dilation in $y$-direction and one dilation in
    $z$-direction. Even this part of the proof is questionable because the
    shape of $R^*$ is not discussed in the proof. However, more importantly,
    the author claims that the proof for next dilations can be obtained by a
    quite similar argument. This is far from being obvious, because the first steps
    presumable create undilatable cubes outside the bounding box but the region
    $R^*$ does not seem to be updated after the first two dilations. In
    addition, Claim~4 claims a stronger statement that there is no undilatable
    cube outside of the bounding box without a proof (only stating that this is
    the same as the previous proof). This is very confusing, why didn't the
    author prove the stronger statment rather than the weaker one. In addition,
    ``undilatable cubes'' outside the original bounding box seem to exist after
    a sequence of dilations; see a sketch in Figure~\ref{f:undilatable}. (We do
    not provide a full 3-dimensional counterexample but the Figure sketches a
    problem that is not discussed neither in the proof of Claim~4 nor in the
    proof of Claim~3.
    \begin{figure}
      \begin{center}
	\includegraphics[page=21]{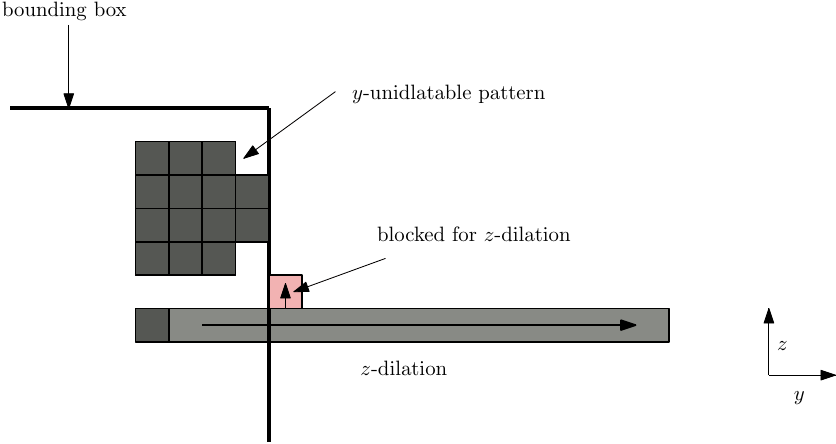}
	\caption{Consider a cut through the animal by a plane perpendicular to
	the $x$-axis. The first dilation is performed in the $y$-direction.
	After this dilation, some cube outside the $z$-direction may be blocked
	due to some $y$-undilatable pattern.}
	\label{f:undilatable}
      \end{center}
    \end{figure}
 \item Another problem seems to be the structure of the proof of Claim 5. The
   aim of the claim is to show that after a finite number of dilations, the
    number of non-applicable cubes is reduced. The author aims to show some
    non-applicable cube can be changed into an applicable one. However, a
    $y$-dilation of some cube may create new $z$-undilatable cubes. (This is
    not discussed in the proof and the $y$-dilation is not forbidden in this
    case.) An example when this may occur is sketched in a $2$-dimensional cut
    in Figure~\ref{f:blocking}.
    \begin{figure}
      \begin{center}
	\includegraphics[page=22]{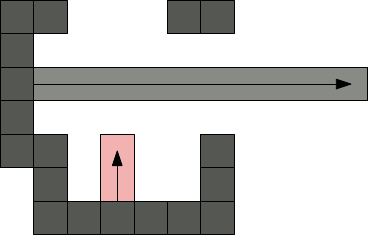}
	\caption{A $y$-dilation blocking a $z$-dilation.}
	\label{f:blocking}
      \end{center}
    \end{figure}

 \item On page 6, below Figure 6, the author claims that it is sufficient to
   avoid undilatable patterns and then it is sufficient to know that the
   problem for simple deformations is solvable, referring
   to~\cite{klette-rosenfeld04}. We point out that~\cite{klette-rosenfeld04} is
   a book with $\sim 600$ pages, thus it may be difficult to understand where
   the author refers to. In particular, \cite{klette-rosenfeld04} seems to
   discuss how to perform a magnification of the image via simple deformations
   in dimension $2$ but not in dimension $3$. This is a huge difference. It is
   possible that the author intends to refer to the
   manuscript~\cite{nakamura-morita-imai06}. But even in this case, this claim
   would deserve more details. Presumably, even if the animal contains no
   undilattable patterns, the may be created during the simple deformations.
\end{enumerate}

\end{document}